\documentclass[
	11pt,oneside,british,a4paper]{amsart}
\usepackage[T1]{fontenc}
\usepackage{enumitem}
\usepackage[utf8]{inputenc}
\usepackage[dvipsnames]{xcolor}
\usepackage[numbers]{natbib}
\usepackage{amstext,amsthm,amssymb}

\usepackage{amsrefs}
\usepackage{enumitem}
\usepackage{dsfont}
\usepackage{xparse,mathrsfs,mathtools}
\usepackage{bookmark}
\usepackage[capitalise,nameinlink]{cleveref}
\usepackage{babel}
\usepackage[yyyymmdd,hhmmss]{datetime}
\usepackage{hyperref}
\usepackage{lmodern,microtype}
\usepackage{tikz}
\usepackage[a4paper, total={6.5in, 8.5in}]{geometry}
\hypersetup{
	unicode=true,%
	pdfdisplaydoctitle=true,%
	pdfpagemode=UseOutlines,%
	breaklinks=true,%
	pdfencoding=unicode,psdextra,
	pdfcreator={},
	pdfborder={0 0 0},
	hidelinks
}
\newtheorem{thm}{Theorem}[section]
\newtheorem{prop}[thm]{Proposition}

\newtheorem{lem}[thm]{Lemma}
\newtheorem{cor}[thm]{Corollary}
\newtheorem{rem}[thm]{Remark}

\newcommand{\NN}{\mathbb{N}}
\newcommand{\ZZ}{\mathbb{Z}}

\newcommand{\RR}{\mathbb{R}}

\newcommand{\kw}{\psi_{k,K_0\mu}}

\newcommand{\kww}{\psi_{k,K_0\mu'}}

\newcommand{\norm}[1]{\left\lVert#1\right\rVert}
\newcommand{\fspq}{F^s_{p,q}}
\newcommand{\fspqq}{F^{-1/q'-n}_{p,q}}

\DeclarePairedDelimiter{\abs}{\lvert}{\rvert}

\usepackage{xifthen,ifthen}
\makeatletter
\newcounter{@ToDo}
\newcommand{\todo@helper}[1]{%
	({\color{blue}TODO~\arabic{@ToDo}: {#1\@addpunct{.}}})%
}
\newcommand{\todo}[1]{\stepcounter{@ToDo}%
	\relax\ifmmode\text{\todo@helper{#1}}%
	\else\todo@helper{#1}\fi%
}

\title[Orthogonal Systems of Spline Wavelets as Unconditional Bases]
{Orthogonal Systems of Spline Wavelets as Unconditional Bases in Sobolev Spaces}
\date{\today}

\author{Rajula~Srivastava}
\address{Rajula Srivastava\\
Department of Mathematics\\ 
University of Wisconsin-Madison\\
480 Lincoln Dr, Madison\\
WI-53706\\
USA}
\email{rsrivastava9@wisc.edu}

\begin{document}
\begin{abstract}
	We exhibit the necessary range for which functions in the Sobolev spaces $L^s_p$ can be represented as an unconditional sum of orthonormal spline wavelet systems, such as the Battle-Lemari\'e wavelets. We also consider the natural extensions to Triebel-Lizorkin spaces. This builds upon, and is a generalization of, previous work of Seeger and Ullrich, where analogous results were established for the Haar wavelet system.
\end{abstract}
\maketitle

\section{Introduction}
It is well known that unlike the trigonometric system, the Haar system forms an unconditional basis in $L^p[0,1]$ for all $1<p<\infty$ (see [\cite{marcinkiewicz1937quelques}). In this article, we aim to explore the analogous problem in the case of Sobolev (and Triebel-Lizorkin) spaces. More precisely, we seek to answer the following question: for what Sobolev spaces does a given orthonormal spline wavelet system form an unconditional basis? 

Let $n\in \NN\cup \{0\}$. We consider an orthogonal spline system on the real line, characterized by a scaling function $\Psi_n$ and an associated wavelet $\psi_n$ (both real valued) with the following properties:
\begin{enumerate}[label=(\Alph*)]
    \item \label{prop smoothness} $\Psi_n,\psi_n \in C^{n-1}(\RR)$ (no condition for $n=0$).
    \item \label{prop restriction poly} The restriction of $\Psi_n,\psi_n$ to each interval $\left(j,j+\frac{1}{2}\right)$ (for $j\in \ZZ/2$) is a polynomial of degree at most $n$.
    \item \label{prop exp decay}
    When $n>0$, there exist positive constants $C$ and $\gamma$ (depending on $n$) such that \[
    \abs{\Psi^{(\alpha)}_n(x)}+\abs{\psi^{(\alpha)}_n(x)}\leq Ce^{-\gamma\abs{x}} \text{ for all } 0\leq \alpha\leq n-1.    
    \]
    \item \label{prop moment cancel}
    \[\int x^{M}\psi_n(x)\,dx=0, \text{ for } M=0,1,\ldots,n.\]
\end{enumerate}
We say that $\psi_n$ is of order $n$.
When $n=0$, the Haar wavelet is perhaps the simplest and the most famous example of this type, with 
\[\Psi_0(x)=\mathds{1}_{[0,1]},\;\psi_0(x)=\mathds{1}_{[0,\frac{1}{2}]}(x)-\mathds{1}_{[\frac{1}{2},1]}(x),\]
where $\mathds{1}_{[a,b]}$ denotes the characteristic function of the interval $[a,b]$. More generally, for $n\geq 0$, the Battle-Lemari\'e wavelets (constructed independently by Battle [\cite{battle1987block} and Lemari\'e [\cite{lemarie1988ondelettes}, also investigated by Mallat [\cite{mallat1989multiresolution}) are well-known examples of such a system. 

For $k\in \NN\cup\{0\}$ and $\mu\in\ZZ$, we define \[\psi_{n,k,\mu}:=\psi_n(2^k\cdot-\mu)\;\text{ and }\; \psi_{n,-1,\mu}:=\sqrt{2}\Psi_n(2^k\cdot-\mu).\]
The Battle-Lemari\'e wavelets form an example of what is known as a multiresolution analysis in wavelet theory. We refer the interested reader to standard texts like [\cite{daubechies1992ten}, Section 5.4 and [\cite{wojtaszczyk1997mathematical}, Section 3.3 for a more thorough discussion and actual construction of these wavelet systems. For our purposes, it is sufficient to know that they satisfy the properties \ref{prop smoothness}-\ref{prop moment cancel} above. One must think of the Battle-Lemari\'e system of order $n$ as an "orthonormalized" wavelet version of the $n$-th order cardinal spline $N_n$, recursively defined by the relation $N_0=\mathds{1}_{[0,1]}$, and
\[N_n(x)=(N_{n-1}*N_1)(x),\]
for $n\geq 1$.
In particular, the system 
\begin{equation}
\label{wavelet system}    
    \mathcal{W}_n:=\{2^{k/2}\psi_{n,k,\mu}:\,k\in \NN\cup\{-1\},\mu\in \ZZ\}
\end{equation}
forms an orthornormal basis in $L^2(\RR)$.

We remark here that there also exist other (non-orthogonal) wavelet systems which generalize the idea of B-splines, such as
\begin{itemize}
    \item Chui-Wang wavelets: These wavelets, constructed independently by Chui-Wang [\cite{chui1992compactly} and Unser-Aldroubi-Eden [\cite{unser1992asymptotic}, retain inter-scale orthogonality and are compactly supported.
    \item Bi-orthogonal wavelets: Introduced by Cohen-Daubechies-Feauveau [\cite{cohen1992biorthogonal}, these wavelets are compactly supported, symmetric and regular, but non-orthogonal, with a dual basis generated by another compactly supported wavelet.
\end{itemize}
We refer the reader to [\cite{unser1997ten} for a concise introduction and comparison. In this article, we will focus on orthogonal wavelet systems, although it might be possible to adapt some of our ideas  to the aforementioned systems as well.

Triebel ([\cite{triebel1978haar}, [\cite{triebel2010bases}) showed that the Haar system forms an unconditional basis in Besov spaces $B^s_{p,q}$ if $1<p,q<\infty$ and $-1/p'<s<1/p$ (also see [\cite{ropela1976spline}). The endpoint case $s=1/p$ (and the dual case $s=-1/p'$) can be excluded by noting that all the Haar functions belong to $B^{1/p}_{p,q}$ if and only if $q=\infty$. 

In the case of Sobolev and Triebel-Lizorkin spaces, we have a dependence on the secondary integrability parameter $q$ as well. More precisely, it was shown by Triebel [\cite{triebel2010bases} that the Haar system forms an unconditional basis in the Sobolev spaces $L^s_{p} (1<p<\infty)$ when $\max{\{-1/p',-1/2\}}<s<\min{\{1/p,1/2\}}$ (recall that the norm in $L^s_p$ is given by $\norm{f}_{L^s_p}=\norm{D^sf}_{L^p}$ where $D^sf=\mathcal{F}^{-1}[(1+|\xi|^2)^{s/2}\hat{f}(\xi)])$. It had been an open question if the Haar system formed an unconditional basis in $L^s_p$ in the range $1/2\leq s\leq 1/p$ (for $1<p<2$) and $-1/p'\leq s\leq -1/2$ (when $2<p<\infty$). This was answered in the negative in [\cite{seeger2017haar}, where Seeger and Ullrich established that the aforementioned sufficient condition is also a necessary one. In fact, in [\cite{seeger2017haar}, the question was settled for the general class of Triebel-Lizorkin spaces $F^s_{p,q}$ (we recall that by Littlewood-Paley theory, $L^s_p=F^s_{p,2}$ for $s\in \RR$ and $1<p<\infty$). In a series of follow up papers, Garrig\'os, Seeger and Ullrich also established slightly better necessary and sufficient ranges for suitable enumerations of the Haar system to form a Schauder basis in Besov and Triebel-Lizorkin spaces (see [\cite{garrigos2018haar}), including the limiting case for the former (in [\cite{garrigos2019basis}) and the endpoint case for the latter (see [\cite{garrigos2019haar}).

It is clear from the above discussion that the Haar system is not a good candidate for an unconditional basis in function spaces of higher order smoothness. This is because the Haar wavelet has poor regularity (it fails to be even continuous). Hence, we turn our attention to orthonormal spline wavelet systems satisfying properties \ref{prop smoothness}-\ref{prop moment cancel}. For such systems, Bourdaud [\cite{bourdaud1995ondelettes}  and Triebel [\cite{triebel2010bases} proved results analogous to the Haar case for Besov spaces, with a shift in the range of the smoothness parameter domain corresponding to the shift in regularity of the basis functions. More precisely, they proved that the system $\mathcal{W}_n$ forms an unconditional basis in $B^s_{p,q}$ if $1<p,q<\infty$ and $-n-1/p'<s<n+1/p$. This range is also optimal, for $\psi_n \notin B^s_{p,q}$ for $s\geq n+1/p$ or for $s\leq -n-1/p'$. 

Coming to the case of the Sobolev and Triebel-Lizorkin spaces, Triebel ([\cite{triebel2010bases}, Theorem 2.49, (ii)) showed that the system $\mathcal{W}_n$ (generated by a spline wavelet $\psi_n$ of order $n$, satisfying properties \ref{prop smoothness}-\ref{prop moment cancel}) forms an unconditional basis in  $F^s_{p,q} (1<p,q<\infty)$ when
\begin{equation}
\max{\{-1/p',-1/q'\}}-n<s<\min{\{1/p,1/q\}}+n   
\end{equation} 
(see [\cite{sickel1990spline} for related results for splines). It is an open question if in this case too, the aforementioned sufficient condition is also necessary and in particular, whether the system $\mathcal{W}_n$ is an unconditional basis on the Sobolev space $L^s_p$ for the ranges $1<p<2$, $n+1/2\leq s\leq n+1/p$ and $2<p<\infty$, $-n-1/p'\leq s\leq -n-1/2$ (see figure \ref{fig:my_label}). We answer this question in the negative.

Our main result is the following:
\begin{thm}
\label{main qual thm}
Let $n\in \NN\cup\{0\}$ and $1<p,q<\infty$. The system $\mathcal{W}_n$ (as defined in \eqref{wavelet system}) is an unconditional basis in $F^s_{p,q}$ only if 
\[\max\{-1/p',-1/q'\}-n<s<\min\{1/p,1/q\}+n.\]
\end{thm}
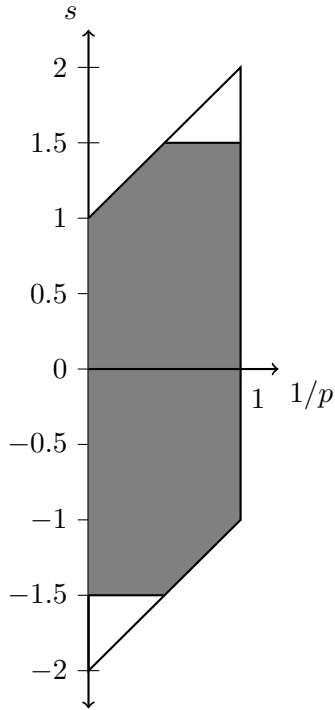
\begin{figure}
    \centering
    \begin{tikzpicture}[scale=2]
\draw[step=1cm,thick,->] (0,0) -- (1.25,0) node[anchor=north west] {$1/p$};
\draw[step=1cm,thick,<->] (0,-2.25) -- (0,2.25) node[anchor=south east] {$s$};
\foreach \x in {0.5,1}
   \draw (\x cm,2pt) -- (\x cm,-2pt) node[anchor=north west] {$\x$};
\foreach \y in {-2,-1.5,-1,-0.5,0,0.5,1,1.5,2}
    \draw (2pt,\y cm) -- (-2pt,\y cm) node[anchor=east] {$\y$};
\draw[thick,fill=gray] (0,1) -- (0.5,1.5)--(1,1.5)--(1,-1)--(0.5,-1.5)--(0,-1.5); 
\draw[thick] (0.5,1.5)--(1,2)--(1,1.5);
\draw[thick] (0.5,-1.5)--(0,-2)--(0,-1.5);

\draw[thick] (0,0) -- (1,0); 
\end{tikzpicture}
    \caption{Domain for an unconditional basis in $L^s_p$ spaces for spline wavelets of order 1.}
    \label{fig:my_label}
\end{figure}

\begin{rem}
Since the Haar system corresponds to the case $n=0$, our result is a generalization of the one in [\cite{seeger2017haar} to orthogonal spline wavelet systems of arbitrary order.
\end{rem} 

Following [\cite{seeger2017haar}, we introduce a suitable framework to quantify the failure of unconditional convergence. For $k\geq 0$, we define the spline wavelet frequency of $\psi_{n,k,\mu}$ to be $2^k$. For any subset $E$ of the system $\mathcal{W}$ let $SF(E)$ denote the spline wavelet frequency set of $E$. In other words,
\[SF(E)=\{2^k:k\geq 0, \text{ there exists }\mu\in \ZZ \text{ with } \psi_{n,k,\mu}\in E\}.\]
We denote by $P_E$ the orthogonal projection to the subspace spanned by $\{g:g\in E\}$ (which is closed in $L^2(\RR))$. In particular, for a Schwartz function $f$,
\[P_Ef=\sum_{\psi_{n,k,\mu}\in E}\langle f,\psi_{n,k,\mu}\rangle\psi_{n,k,\mu}.\] Now for any $A\subset \{2^j:j=0,1,\ldots\}$, we set
\begin{equation}
    \mathcal{G}(\fspq, A)=\sup\{\|P_E\|_{F^s_{p,q}}:SF(E)\subset A\}.
\end{equation}
For $\Lambda\in \NN$, we define the lower wavelet projection number
\begin{equation}
    \label{eq lower proj numbers}
    \gamma_*(\fspq, \Lambda)=\inf\{\mathcal{G}(
    \fspq, A):\#A\geq \Lambda\}.
\end{equation}
As $\psi_{n,k,\mu}\notin F^s_{p,q}$ for $s\geq 1/p+n$, we have that $\gamma_*(\Lambda)=\infty$. By duality, $\gamma_*(\Lambda)=\infty$ for $s\leq -n+1/p'$. In our discussion throughout, we shall assume that $\Lambda>10$. 

The approach used in [\cite{seeger2017haar} to establish the necessary range for unconditional convergence in the case of the Haar basis was the quantification of the growth rate $\mathcal{G}(\fspq, A)$ in terms of the cardinality of $A$. In particular, to give precise lower bounds for $\gamma_*(\fspq, \Lambda)$, the authors constructed a suitable test function, by first considering a sum of the translates of a smooth compactly supported function $\eta$ at a fixed dyadic scale and then taking a randomized sum of the functions hence constructed at different scales, dictated by the frequency of the given set $A$. The Sobolev (or Triebel-Lizorkin) norm of the test function was controlled by introducing enough separation between the translates at the same scale (Proposition \ref{prop tech}). This separation and the compact support of the Haar wavelet was also used to ensure that at each scale, a given translate of the wavelet interacted with exactly one translate of the test function. Finally, by choosing $\eta$ to be an odd function and exploiting the anti-symmetry of the Haar wavelet (with respect to the midpoint of the associated interval), the authors were able to avoid cancellation and get the different interactions to add up, yielding the desired lower bounds.

In this paper, we use the same example as above, and verify that this approach also works for the wavelet systems we consider. In [\cite{seeger2017haar}, the authors had the advantage of working with the Haar wavelet, which can be written down in a very simple closed form and is compactly supported. However, in our article we do not use any explicit formulas for the wavelets (which can get very tedious as the order increases). Neither are our wavelets compactly supported. The novelty of this paper lies in identifying the properties hidden behind the deceptively simple form of the Haar wavelet, which make the example in [\cite{seeger2017haar} work, and adapting them to our setting. Moreover, exponential decay (property \ref{prop exp decay}) is only slightly worse than being compactly supported, and can be essentially dealt with by increasing the separation between the different translates. Consequently, we obtain some tail terms (absent in [\cite{seeger2017haar}), which need to be carefully considered. 

\subsection*{Notation} We shall use the notation $A\lesssim B$, or $B\gtrsim A$, if $A\leq CB$ for a positive constant depending only on $p,q,s$ and the wavelet $\psi_n$ under consideration. Also, if both $A\lesssim B$ and $B\lesssim A$, we shall use the notation $A\approx B$.

\subsection*{Plan of the paper}
In Section \ref{sec fspq spaces}, we briefly discuss the characterization of Triebel-Lizorkin spaces via compactly supported local means, which is quite suitable for our problem. In Section \ref{sec wavelet properties}, we reformulate the properties of the orthogonal spline wavelets in a quantitative form. In Section \ref{sec test func 1}, we state and prove a technical lemma. This is in preparation for defining a suitable family of test functions in $\fspq$, which we do for $p>q$ and $s\leq -1/q'-n$ in Section \ref{sec test func}. In Section \ref{sec prelim estimates}, we establish a few preliminary estimates and lower bounds for the interactions of the test functions with the members $\psi_{n,k,\mu}$ of the wavelet family. In Sections \ref{sec lower bounds} and \ref{sec end point}, we prove the existence of the desired lower bounds for $\gamma_*(\fspq,\Lambda)$ when $s<-1/q'$ and $s=1/q'$, respectively.

\subsection*{Acknowledgements} The author would like to thank her advisor Andreas Seeger for introducing this problem, for his guidance and several illuminating discussions. She is also grateful to Dr. Nadiia Derevianko and Prof. Tino Ullrich for sharing a copy of their manuscript [\cite{DerUll19}, which helped shape some of the ideas contained here. Research supported in part by NSF grant 1500162.

\section{Some Background on Triebel-Lizorkin Spaces}
\label{sec fspq spaces}
We briefly discuss the characterization of Triebel-Lizorkin spaces via "local means" (termed so in [\cite{triebeltheory}, Section 2.4.6) which will be useful for our purposes. 

The usual way to define Triebel-Lizorkin spaces is via a smooth dyadic decomposition of unity. Let $\varphi_0$ be a smooth function supported in $[-3/2,3/2]$ such that $\varphi_0\equiv 1$ on $[-4/3,4/3]$. We set $\varphi=\varphi_0(\cdot)-\varphi_0(2\cdot)$, so that $\varphi_0+\sum_{k\in \NN}\varphi_0(2^{-k}\cdot)\equiv 1$. Defining $\widehat{L_0f}=\varphi_0\hat{f}$ and $\widehat{L_kf}=\varphi(2^{-k}\cdot)\hat{f}$ for a Schwartz function $f$, we obtain an inhomogenous Littlewood-Paley decomposition
\begin{equation}
\label{littlewood-paley}    
f=\sum_{k=0}^{\infty} L_Kf,
\end{equation}
with convergence in $\mathcal{S}'(\RR)$ and in all $L^p$ spaces. For $0<p<\infty, \,0<q\leq \infty$, and $s\in \RR$, the Triebel-Lizorkin space $\fspq(\RR)$ is defined as the collection of all tempered distributions $f\in \mathcal{
S}'(\RR)$ such that
\begin{equation}
\label{triebel lizorkin definition}  
\|f\|_{\fspq}=\bigg\|\bigg(\sum_k 2^{ksq} |L_kf|^q\bigg)^{1/q}\bigg\|_{L^p}<\infty,
\end{equation}
with the usual modification when $q=\infty$.
We now define another pair of functions $\phi_0$ and $\phi$ such that $|\widehat{\phi_0}|>0$ on $(-\epsilon,\epsilon)$ and $|\widehat{\phi}|>0$ on the set $\{\xi: \epsilon/4<|\xi|<\epsilon\}$. We also assume that
\begin{equation}
\label{eq local means moment cancel}
\int \phi(y)y^m\,dy=0,    
\end{equation}
for $m=0,\ldots, M_1$ where $M_1\in \NN$ is such that $M_1+1>s$. It can be proved using vector valued singular integrals (see [\cite{triebeltheory}, Section 2.4.6) that
\begin{equation}
\label{local means char}    
\|f\|_{\fspq}\approx \bigg\|\bigg(\sum_k 2^{ksq} |\phi_k*f|^q\bigg)^{1/q}\bigg\|_{L^p},
\end{equation}
with $\phi_k(x)=2^k\phi(2^x)$. The above characterization allows for compactly supported $\phi$ and $\phi_0$, termed as "local means". 

\section{Properties of the Orthogonal Spline Wavelets}
\label{sec wavelet properties}
For $k,\mu\in\ZZ$, we define
\begin{equation}
I_{k,\mu}:=[2^{-k}\mu, 2^{-k}(\mu+1)],   
\end{equation}
and
\begin{equation}
x_{k,\mu}:=2^{-k}\mu + 2^{-k+1}.   
\end{equation}
The Haar wavelet generates a system that can be easily written down explicitly. Unfortunately, these formulas become extremely complicated when $n>0$. Moreover, $\psi_n$ is no longer compactly supported in this case. However, on a closer inspection, one can isolate the primary properties of the Haar system on which the arguments in [\cite{seeger2017haar} are based. These are: 
\begin{enumerate}[label=(\alph*)]
    \item Each $\psi_{0,k,\mu}$ is supported on the dyadic interval $I_{k,\mu}$.
    \item $\psi_{0,k,\mu}$ is anti-symmetric around $x_{k,\mu}$.  
\end{enumerate}
The test functions are then constructed by taking a sum of compactly supported functions $\eta_{k,\mu}$ centred around $x_{k,\mu}$ for $0\leq \mu\leq 2^k-1$. The first property ensures enough separation so that each wavelet translate $\psi_{0,k,\mu_0}$ interacts with exactly one translate of $\eta$ at scale $2^{-k}$, namely $\eta_{k,\mu_0}$. The second property is exploited by considering $\eta_{k,\mu}$ to be \textit{odd}, so that the contributions from both halves of $I_{k,\mu}$ get added up.  

In our case, even though $\psi_n$ does not have compact support, it is only slightly worse: $\psi_{n,k,\mu}$ (and its derivatives) decay exponentially off of $I_{k,\mu}$ (property \ref{prop exp decay}). Thus, by introducing enough separation (as determined by the decay rate), we can still ensure that the interaction of $\eta_{k,\mu}$ with $\psi_{n,k,\mu'}$ is negligibly small when $\mu\neq \mu'$. 

Speaking of symmetry, although the Battle-Lemari\'e wavelet of order $n$ is known to be symmetric (anti-symmetric) around $1/2$ when $n$ is odd (even), we don't rely on this property in our argument, in order to make it applicable to general settings. Let us consider the unit interval $[0,1]$ (for the other dyadic intervals can be obtained from this case by appropriate scaling and translation). $\psi_n$ is represented by (different) polynomials of degree $n$ on $[0,1/2]$ and $[1/2,1]$. However, the condition that $\psi_n\in C^{n-1}$ forces the non-leading left and right co-efficients to be equal. This takes care of the cancellation of the lower order polynomial terms, provided the lower moments of the test function disappear, which is indeed chosen to be so by construction. Further, by considering a translation of $\psi_n$, if necessary, we can assume that the leading co-efficients of the left and right polynomial representation of $\psi_n$ around $1/2$ are not equal. Then by choosing a test function $\eta$ such that $y^n\eta(y)$ is odd, we can still get the interactions to add up, yielding non-zero lower bounds. In the endpoint case, we use a slight generalization of this idea, choosing $\eta$ to be even or odd depending on the signs of the leading co-efficients of $\psi_n$ around $0$ with respect to each other.

In the paper henceforth, $n\in \NN$ shall remain fixed and be understood from the context. Consequently, we denote $\psi_n(x)$ by $\psi(x)$ and $\psi_{n,k,\mu}$ by $\psi_{k,\mu}.$ The following lemma is a quantitative interpretation of Properties \ref{prop smoothness} and \ref{prop exp decay}. 
\begin{lem}
\label{lemma wavelet prop}
Let $\theta\in\ZZ$. Suppose $\psi$ is represented by
\begin{equation}
\label{left wave form}
\psi(x)=A^n_{\theta-1}\left(x-\frac{\theta}{2}\right)^n+A^{n-1}_{\theta-1}\left(x-\frac{\theta}{2}\right)^{n-1}+\ldots+A^0_{\theta-1}
\end{equation}
on $[\frac{\theta-1}{2}, \frac{\theta}{2}]$, and
\begin{equation}
\label{right wave form}
\psi(x)=A^n_{\theta}\left(x-\frac{\theta}{2}\right)^n+A^{n-1}_{\theta}\left(x-\frac{\theta}{2}\right)^{n-1}+\ldots+A^0_{\theta}
\end{equation}
on $[\frac{\theta}{2}, \frac{\theta+1}{2}]$. Then we have
\begin{enumerate}[label=(\roman*)]
    \item \[A^{j}_{\theta-1}=A^{j}_{\theta}\] for $j=0,1,\ldots,n-1$.
    \label{coeff lower}
    \item \[|A^n_\theta|\leq 4Ce^{\gamma/2}e^{-\gamma\abs{\frac{\theta}{2}}},\] where $C$ and $\gamma$ are as defined in Property \ref{prop exp decay}.
    \label{coeff decay}
\end{enumerate}
\end{lem}
\begin{proof}
By virtue of the fact that $\psi\in C^j$ for $j=0,1,\ldots,n-1$ (Property \ref{prop smoothness}), we have that
\[\lim_{x\rightarrow \theta/2^-}\psi^{(j)}{(x)}=\lim_{x\rightarrow \theta/2^+}\psi^{(j)}{(x)},\] which yields
\[A^j_{\theta-1}=A^j_{\theta}\]
for $0\leq j\leq n-1$, thus proving \ref{coeff lower}.


For proving \ref{coeff decay}, we use Property \ref{prop exp decay} with $\alpha=n-1$ to obtain
\begin{equation}\label{coeff decay est1}
\bigg|n!\, A^n_\theta\bigg(x-\frac{\theta}{2}\bigg)+(n-1)!\, A^n_{\theta-1}\bigg|\leq Ce^{-\gamma \abs{x}}
\end{equation}
for $x\in \big[\frac{\theta}{2},\frac{\theta+1}{2}\big]$. In particular, taking the limit as $x\rightarrow\theta/2$, we get
\begin{equation}
\label{coeff decay est2}
|A^n_{\theta-1}|\leq \frac{C}{(n-1)!}e^{-\gamma \abs{\frac{\theta}{2}}}.    
\end{equation}
Now, we substitute $x=\frac{\theta+1}{2}$ in \eqref{coeff decay est1} and use the triangle inequality, along with \eqref{coeff decay est2}, which yields
\[\bigg|\frac{A^n_\theta}{2}\bigg|\leq \frac{C}{n!}(e^{-\gamma\abs{\frac{\theta+1}{2}}}+e^{-\gamma\abs{\frac{\theta}{2}}})\leq 2Ce^{\gamma/2}e^{-\gamma\abs{\frac{\theta}{2}}},\]
which implies \ref{coeff decay}.
\end{proof}

\section{Boundedness of Test Functions}
\label{sec test func 1}
We now prepare the ground for the definition of test functions to be used to establish the desired lower bounds. The arguments used in this section are identical to those in [\cite{seeger2017haar}, Section 4. Nevertheless, we include them here for completeness. Throughout this section, we fix $m\in \NN$. 

We will use the local means characterization of Triebel-Lizorkin spaces, as described in Section \ref{sec fspq spaces}. To this effect, we consider smooth functions $\phi_0$ and $\phi$, both supported in $(-1/2,1/2)$ so that $|\widehat{\phi_0}(\xi)|>0$ for $|\xi|\leq 1$ and $|\widehat{\phi}(\xi)|>0$ for $1/4\leq |\xi|\leq 1$. We also assume that the cancellation condition \eqref{eq local means moment cancel} holds for $\phi$, for $M_1\in \NN$ with $M_1+1>s$. We set $\phi_k=2^k\phi(2^k\cdot)$ for $k\in \NN$. We shall use the characterization of $\fspq$ using the $\phi_k$, as defined in \eqref{local means char}. 

Let $\eta\in C^\infty(\RR)$ be supported in $(-1/2,1/2)$ such that $\int x^M\eta(x)\,dx=0$ for $M=0,1,\ldots, n+2$. Let $\mathfrak{L}^m$ be a finite set of non-negative integers $\geq m$, such that $\#\mathfrak{L}^m \geq 2^m$. For each $l\in \mathfrak{L}^m$, let $\mathcal{P}^m_l$ denote a set of $K_02^{m-l}$ separated points in $[0,K_0]$, where $K_0\in \NN$ is a fixed positive integer to be decided later. More precisely, we have $\mathcal{P}^m_l=\{x_{l,1},\ldots, x_{l,N(l)}\}$ with $N(l)\leq 2^{l-m}$ and $x_{l,\nu}<x_{l,\nu+1}$ with $x_{l,\nu+1}-x_{l,\nu}\geq K_02^{m-l}$. For each $l\in \mathfrak{L}^m$, let 
\begin{equation}
\mathfrak{S}^m_l=\{\nu: x_{l,\nu}\in \mathcal{P}^m_l\}.    
\end{equation}
Next, we define
\begin{equation}
    \eta_{l,\nu}=\eta(2^l(x-x_{l,\nu})).
\end{equation}
For a sequence $\{a_{l,\nu}\}$ with $\sup_{l,\nu}\abs{a_{l,\nu}}\leq 1$, we define
\begin{equation}
g_m(x)=\sum_{l\in \mathfrak{L}^m}2^{-ls}\sum_{\nu\in \mathfrak{S}^m_l}a_{l,\nu}\eta_{l,\nu}(x).    
\end{equation}
If the families $\mathfrak{L}^m$, $(m\in \NN)$ are disjoint, we define
\begin{equation}
    g= \sum_{m\in \NN}\beta_m g_m
\end{equation}
for $\beta_m\in \RR$.

The following proposition is identical to Proposition 4.1 in [\cite{seeger2017haar} (which was in turn a modification of the corresponding result in [\cite{christ2006necessary}).

\begin{prop}
\label{prop tech}
\leavevmode
Let $s>-n-2$. 
\begin{enumerate}[label={\upshape(\roman*)}]
    \item \label{prop tech 1}
    For $1\leq p,q<\infty$, we have
    \[\|g_m\|_{\fspq}\lesssim_{p,q,s}\bigg\|\bigg(\sum_{l\in \mathfrak{L}^m}\big|\sum_{\nu\in \mathfrak{S}^m_l}a_{l,\nu}\mathds{1}_{l,\nu}\big|^q\bigg)^{1/q}\bigg\|_p\]
    and
    \[\|g\|_{\fspq}\lesssim_{p,q,s}\bigg\|\bigg(\sum_{m\in \NN}|\beta_m|^q\sum_{l\in \mathfrak{L}^m}\big|\sum_{\nu\in \mathfrak{S}^m_l}a_{l,\nu}\mathds{1}_{l,\nu}\big|^q\bigg)^{1/q}\bigg\|_p.\]
    Here, $\mathds{1}_{l,\nu}$ denotes the characteristic function of the interval centred at $x_{l,\nu}$ of radius $2^{-l}$.
    \item \label{prop tech 2}
    For $1\leq q\leq p,\infty$, we have
    \begin{equation}
    \label{eqn tech 1}
     \|g_m\|_{\fspq}\lesssim_{p,q,s}(2^{-m}\#(\mathfrak{L}^m))^{1/q}   
    \end{equation}
    and
    \begin{equation}
    \label{eqn tech 2}
     \|g\|_{\fspq}\lesssim_{p,q,s}(\sum_{m\geq 1}|\beta_m|^q2^{-m}\#(\mathfrak{L}^m))^{1/q}.   
    \end{equation}
\end{enumerate}
\end{prop}
\begin{proof}
\ref{prop tech 1} is a consequence of the fact that $\{\eta_{l,\nu}\}_{l,\nu}$ form a family of smooth atoms in the sense of Frazier and Jawerth ([\cite{frazier1990discrete}, Theorem 4.1 and Section 12). We use the pairwise disjointednes of the sets $\mathfrak{L}^m$ here. 

Consequently, in order to prove \eqref{eqn tech 1}, it suffices to show that
\begin{equation}
\label{eq count bound}
    \bigg\|\bigg(\sum_{l\in \mathfrak{L}^m}\big|\sum_{\nu\in \mathfrak{S}^m_l}\mathds{1}_{l,\nu}\big|^q\bigg)^{1/q}\bigg\|_p\lesssim_{p,q}(2^{-m}\#(\mathfrak{L}^m))^{1/q}
\end{equation}
(recall that $\sup_{l,\nu}|a_{l,\nu}|\leq 1$).
Let $G_l(x)=\sum_{\nu\in \mathfrak{S}^m_l}\mathds{1}_{l,\nu}(x)$ and $G(x)=(\sum_{l\in \mathfrak{L}^m}|G_l(x)|^q)^{1/q}$. To prove the desired inequality for the $L^p$ norm of $G$, we use the dyadic version of the Fefferman-Stein interpolation theorem for $L^q$ and $BMO$ (see [\cite{stein1993harmonic}, Chapter 4). Here we use the fact that $p\geq q$. Thus, it is enough to show that both the $L^q$ and the $BMO_{\text{dyad}}$ norms of $G$ are bounded by $(2^{-m}\#(\mathfrak{L}^m))^{1/q}$. This follows almost immediately for the former. For the $BMO_{\text{dyad}}$ norm, we need to show that
\begin{equation}
\label{eq BMO}
    \sup_J\inf_{c\in \RR}\frac{1}{|J|}\int |G(y)-c|\,dy\lesssim 2^{-m}\#(\mathfrak{L}^m))^{1/q},
\end{equation}
where the supremum is taken over all dyadic intervals $J$. We fix $J$ and denote its midpoint by $x_J$. We define 
\[c_{J,l}=\begin{cases}
\sum_{\nu\in \mathfrak{S}^m_l}\mathds{1}_{I_{l,\nu}}(x_J) &\textrm{ if } |J|\leq 2^{-l}\\
0 &\text{ otherwise,}
\end{cases}\]
and 
\[c_J=(\sum_{l\in \mathfrak{L}^m}c_{J,l}^q)^{1/q}.\]
Then 
\begin{align*}
\frac{1}{|J|}\int |G(y)-c_J|\,dy&= \frac{1}{|J|}\int \bigg|\sum_{l\in \mathfrak{L}^m}|G_l(y)|)^{1/q}-(\sum_{l\in \mathfrak{L}^m}c_{J,l}^q)^{1/q}\bigg|,dy \\
&\leq \frac{1}{|J|}\int \bigg(\sum_{l\in \mathfrak{L}^m}|G_l(y)-c_{J,l}|^q\bigg)^{1/q}\,dy\\
&\leq \bigg(\sum_{l\in \mathfrak{L}^m}\frac{1}{|J|}\int|G_l(y)-c_{J,l}|^q\bigg)^{1/q}\,dy,
\end{align*}
where we have used the triangle inequality in $l^q$ and H\"older's inequality on the interval $J$. Now for $|J|\leq 2^{-l}$ and $y\in J$, we have that $G_l(y)=c_{J,l}$. Also, as $c_{J,l}=0$ for $|J|>2^{-l}$, we get
\begin{equation*}
\frac{1}{|J|}\int |G(y)-c_J|\,dy\leq \bigg(\sum_{l\in \mathfrak{L}^m, 2^{-l}<|J|}\frac{1}{|J|}\int|G_l(y)|^q\,dy\bigg)^{1/q}.     
\end{equation*}
But as the points in $\mathfrak{S}^m_l$ are $K_02^{m-l}$ separated, by the definition of $G_l(y)$, we have
\[\int|G_l(y)|^q\,dy\leq 
\begin{cases}
2^{-l}& \text{ if } 2^{-l}< |J|\leq 2^{m-l}\\
2^{-m}|J|& \text{ if } 2^{m-l}< |J|.
\end{cases}
\]
Hence
\begin{align*}
\sum_{l\in \mathfrak{L}^m, 2^{-l}<|J|}\frac{1}{|J|}\int|G_l(y)|^q\,dy\leq &\sum_{l\:2^{-l}< |J|\leq 2^{m-l}} (2^l|J|)^{-1}+ \sum_{l\in \mathcal{L}^m} 2^{-m} \\
\lesssim (1+2^{-m}\#(\mathfrak{L}^m)).
\end{align*}
This proves \eqref{eq BMO}, as $(\mathfrak{L}^m)\geq 2^m$.

Finally, \eqref{eqn tech 2} can be proven by using the second assertion in \ref{prop tech 1}, \eqref{eq count bound} and the triangle inequality in $L^{p/q}$, noting that $p\geq q$.     
\end{proof}

\section{Definition of the Test Functions for the Non-Endpoint Case}
\label{sec test func}
In this section, we define the test functions to be used to establish the lower bounds in the non-endpoint case. Our example is essentially the same as the one used in [\cite{seeger2017haar}, except we take care to increase the separation between the translates at each dyadic scale (by a factor of $K_0$), to allow the exponential decay of the spline wavelet to kick in. Consequently, our function lives on $[0,K_0]$, rather than the unit interval. We now present the details.

Let $\eta$ be a $C^\infty$ function supported in $(-2^{-5}, 2^{-5})$. We require $\eta$ to be odd for even $n$, and even for odd $n$, so that $x^n\eta(x)$ is always odd. Furthermore, let
\begin{equation}
\label{moment cancel test func}
\int x^M\eta(x)\, dx=0,    
\end{equation} for $M=0,1,\ldots, n+2$ and let
\begin{equation}
\label{testbig}
\int_0^{\frac{1}{2}}x^n\eta(x)\, dx\geq 1.    
\end{equation}

Let $A$ be an arbitrary set of wavelet frequencies and $N$ so that 
\begin{equation}
\label{set freq}
\Lambda< \#A+1 \textrm{ and } 2^N\leq \#A <2^{N+1}.
\end{equation}
$N$ and $\eta$ will remain fixed henceforth. For $k=0,1,2,\ldots$ and $\mu\in \ZZ$, we define
\begin{equation}
\eta_{k,K_0\mu}(y)=\eta(2^{k+N}(y-2^{-k}K_0\mu-2^{-k-1})).
\end{equation}
Let $r_k$ denote the $k$-th Rademacher function on $[0,1]$. For $t\in [0,1]$ and $2^k\in A$ let
\begin{equation}
\label{k-levelfunc}
\Upsilon_k(y)=2^{N(-s+1/q)}\sum_{\mu=0}^{2^k-1}\eta_{k,K_0\mu}(y)
\end{equation}
and 
\begin{equation}
\label{test function}
f_t(y)=2^{-N/q}\sum_{2^k\in A} r_k(t)2^{-ks}\Upsilon_k(y).    
\end{equation}

\begin{lem}
\label{test norm}
$\norm{f_t}_{F^s_{p,q}} \lesssim_{p,q,s} 1 $
uniformly in $t\in [0,1]$.
\end{lem}
\begin{proof}
We write $f_t$ in the expanded form 
\[f_t=\sum_{k:2^k\in A}2^{-(k+N)s}\sum_{\mu=0}^{2^k-1}r_{k}(t)\eta(2^{k+N}(y-2^{-k}K_0\mu-2^{-k-1})).\]

We now set $m=N$, $\mathfrak{L}^N=\{k+N:2^k\in A\}$ and 
apply Proposition \ref{prop tech}, \ref{prop tech 1}. The lemma now follows as $2^{-N}\#(\mathfrak{L}^N)\lesssim 1$ and the points $\{2^{-k}K_0\mu+2^{-k-1}: 0\leq \mu\leq 2^{k}-1\}$ are $K_02^{m-l}$ separated, for $l=k+N$.
\end{proof}
\section{A Few Preliminary Estimates}
\label{sec prelim estimates} In this section we require $\phi$ (as defined in Section \ref{sec test func 1}) to be supported on $(-2^{-4}, 2^{-4})$ such that
\begin{equation}
\label{triebel lizorkin moment cancel}
\int x^M\phi(x)\,dx=0    
\end{equation} for $M=0,1,\ldots,n+1$, and $\|\phi\|_{L^1}\leq 1.$ Let $\phi_k=2^k\phi(2^k\cdot)$. We define $\Phi_1(x)=\int_{-\infty}^{x}\phi(t)\,dt$ and for $j=2,\ldots, n+1$, let \[\Phi_j(x)=\int_{-\infty}^{x}\Phi_{j-1}(t)\,dt\] be the $jth$ order primitive of $\phi$, also supported in $(-2^{-4}, 2^{-4})$. Further, let 
\begin{equation}
\label{left wave form 1/2}
\psi(x)=A^n_{0}\left(x-\frac{1}{2}\right)^n+A^{n-1}\left(x-\frac{1}{2}\right)^{n-1}+\ldots+A^0.
\end{equation}
on $[0, \frac{1}{2}]$, and
\begin{equation}
\label{right wave form 1/2}
\psi(x)=A^n_{1}\left(x-\frac{1}{2}\right)^n+A^{n-1}\left(x-\frac{1}{2}\right)^{n-1}+\ldots+A^0
\end{equation}
on $[\frac{1}{2},1]$, where the equality of the non-leading co-efficients follows from Lemma \ref{lemma wavelet prop}, \ref{coeff lower}. By considering a suitable translation of $\psi_n$ if necessary, we can assume that 
\begin{equation}
\label{leading terms not equal}    
A^n_0\neq A^n_1
\end{equation}
and in particular, that $A^n_0\neq 0$.

\begin{lem}
\label{lemma pou lower bound}
There exists $c_0\in (0,1)$ and a subinterval $J\subset [1/4,3/4]$ so that 
\[\abs{\phi*\psi(x)}\geq c_0\]
for $x\in J$.
\end{lem}
\begin{proof}
We observe that the support of $\phi$ is contained in $[x-1,x]$, whenever $x\in [1/4,3/4]$. For such $x$, we have
\[
\phi*\psi(x)=\int_{x-1/2}^{x} \phi(y)\psi(x-y)\,dy+ \int_{x-1}^{x-1/2} \phi(y)\psi(x-y)\,dy.
\]
We observe that $x-y$ lies in $\big[0,\frac{1}{2}\big]$ in the first integral and 
in $\big[\frac{1}{2},1\big]$ in the second one. Hence we can use (\ref{left wave form 1/2}) and (\ref{right wave form 1/2}) in the left and right integral, respectively. Now, for $j=0,\ldots,n-1$, we have 
\[A^j\int_{x-1/2}^{x}\bigg(x-y-\frac{1}{2}\bigg)^j \phi(y)\,dy+A^j\int_{x-1}^{x-1/2} \bigg(x-y-\frac{1}{2}\bigg)^j\phi(y)\,dy=A^j\int\bigg(x-y-\frac{1}{2}\bigg)^j \phi(y)\,dy.\]
The last expression is easily seen to be $0$ by \eqref{triebel lizorkin moment cancel}. Thus, all the lower degree terms cancel, and we have
\[\phi*\psi(x)=A^n_0\int_{x-1/2}^{x}\bigg(x-y-\frac{1}{2}\bigg)^n \phi(y)\,dy+A^n_1\int_{x-1}^{x-1/2} \bigg(x-y-\frac{1}{2}\bigg)^n\phi(y)\,dy.\]
Now performing an integration by parts $n$ times, along with the observation that the boundary terms are all zero, gives
\[\phi*\psi(x)=(-1)^n n!\,\left [ A_0^n\int_{x-1/2}^{x} \Phi_n(y)\,dy +A_1^n\int_{x-1}^{x-1/2} \Phi_n(y)\,dy\right ].\]
We thus conclude that 
\[\phi*\psi(x)=(-1)^n n!\,\bigg[A_0^n\Phi_{n+1}(x)-A_1^n\Phi_{n+1}(x-1)+(A_1^n-A_0^n)\Phi_{n+1}\big(x-\frac{1}{2}\big)\bigg].\]
In particular, for $x\in [1/4,3/4]$, we have that $\phi*\psi(x)=(-1)^{n+1} n!\,(A_1^n-A_0^n)\Phi\left(x-\frac{1}{2}\right)$. Using \eqref{leading terms not equal}, we conclude that there exists $c_0\in (0,1)$ (depending on $\psi_n$ and $\phi$) and a subinterval $J\subset [1/4,3/4]$ so that 
\[\abs{\phi*\psi(x)}\geq c_0\]
for $x\in J$.
\end{proof}

We again use $K_0$ to denote a fixed positive integer (to be decided later), which shall depend only on the wavelet $\psi_n$. For $k\in \NN\cup\{0\}$ and $\mu\in\ZZ$, let $J_{k,K_0\mu}=2^{-k}K_0\mu+2^{-k}J$ (where $J$ is as in Lemma \ref{lemma pou lower bound}). We then have
\begin{equation}
    \label{convol big}
    \abs{\phi_k*\psi_{k,K_0\mu}(x)}\geq c_0
\end{equation}
for $x\in J_{k,K_0\mu}$. We note that $J_{k,K_0\mu}$ is an interval of length $\gtrsim 2^{-k}$.

\begin{prop}
\label{inner product lower bound}
Let $\Upsilon_k$ be as defined in \eqref{k-levelfunc}. Then for $K_0$ large enough, we have 
\[\frac{|\Tilde{A}|}{2} 2^{N(-s+1/q-n-1)}\left|\int_0^{1/2}\eta(y)y\,dy\right|\leq \abs{2^k\left\langle\Upsilon_k,\psi_{k,K_0\mu}\right\rangle}\leq 2|\Tilde{A}| 2^{N(-s+1/q-n-1)}\left|\int_0^{1/2}\eta(y)y\,dy\right|.\]
Here $\Tilde{A}=A^n_1-A^n_0$ and depends only on the wavelet $\psi_n$.
\end{prop}

\begin{proof}
Using the definition of $\Upsilon_k$, we get
\[2^k\left\langle \Upsilon_k,\psi_{k,K_0\mu}\right\rangle=2^{N(-s+1/q)}\sum_{\mu'=0}^{2^k-1}2^k\left\langle \eta_{k,K_0\mu'},\psi_{k,K_0\mu}\right\rangle.\]
Now, we have 
\begin{align*}
&2^k\left\langle \eta_{k,K_0\mu'},\psi_{k,K_0\mu}\right\rangle
=2^k\int\eta(2^{k+N}(x-2^{-k}K_0\mu'-2^{-k-1}))\psi(2^kx-K_0\mu)\,dx\\
&=\int_0^{1}\eta\bigg(2^N\bigg(y-\frac{1}{2}\bigg)\bigg)\psi\bigg(y+\frac
{\lambda}{2}\bigg)\,dy\\
&=\int_0^{1/2}\eta\bigg(2^N\bigg(y-\frac{1}{2}\bigg)\bigg)\psi\bigg(y+\frac
{\lambda}{2}\bigg)\,dy+\int_{1/2}^{1}\eta\bigg(2^N\bigg(y-\frac{1}{2}\bigg)\bigg)\psi\bigg(y+\frac
{\lambda}{2}\bigg)\,dy,
\end{align*}
where $\lambda=2K_0(\mu'-\mu)$. We observe that $y+\frac{\lambda}{2}$ lies in $\big[\frac{\lambda}{2},\frac{\lambda+1}{2}\big]$ in the first integral and 
in $\big[\frac{\lambda+1}{2},\frac{\lambda+2}{2}\big]$ in the second one. Hence we can use formulations (\ref{left wave form}) and (\ref{right wave form}) of $\psi$ (with $\theta=\lambda+1$), for the left and right integral, respectively. By arguing as in the proof of Lemma \ref{lemma pou lower bound}, using \eqref{moment cancel test func} instead of \eqref{triebel lizorkin moment cancel}, it is easy to see that the lower degree terms cancel out, and we obtain
\begin{multline*}
2^k\left\langle \eta_{k,K_0\mu'},\psi_{k,K_0\mu}\right\rangle=\\ A_{\lambda}^n\int_0^{1/2}\eta\bigg(2^N\bigg(y-\frac{1}{2}\bigg)\bigg)\bigg(y-\frac{1}{2}\bigg)^n\,dy+A_{\lambda+1}^n\int_{1/2}^1\eta\bigg(2^N\bigg(y-\frac{1}{2}\bigg)\bigg)\bigg(y-\frac{1}{2}\bigg)^n\,dy.     
\end{multline*}
Applying a change of variables, we get
\begin{align*}
&2^k\left\langle \eta_{k,K_0\mu'},\psi_{k,K_0\mu}\right\rangle
=A_{\lambda}^n\int_{-1/2}^0y^n\eta(2^Ny)\,dy+A_{\lambda+1}^n\int_0^{1/2}y^n\eta(2^Ny)\,dy\\
&=2^{-(n+1)N}\left(A_{\lambda}^n\int_{-1/2}^0y^n\eta(y)\,dy+A_{\lambda+1}^n\int_0^{1/2}y^n\eta(y)\,dy\right)\\
&=2^{-(n+1)N}(A_{\lambda+1}^n-A_{\lambda}^n)\int_0^{1/2}y^n\eta(y)\,dy,
\end{align*}
where in the last step we have used the fact that $y^n\eta(y)$ is odd.

\noindent When $\mu=\mu'$ and $\lambda=0$, we conclude that
\[2^k\left\langle \eta_{k,K_0\mu},\psi_{k,K_0\mu}\right\rangle=2^{-(n+1)N}\Tilde{A}\int_0^{1/2}y^n\eta(y)\,dy.\]

\noindent When $\mu\neq\mu'$, we take
absolute values and use the exponential decay of the leading co-efficients (part \ref{coeff decay} of Lemma \ref{lemma wavelet prop}) to obtain
\begin{align*}
\abs{2^k\left\langle\eta_{k,K_0\mu'},\psi_{k,K_0\mu}\right\rangle}
&\leq 8C2^{-(n+1)N}e^{\gamma}e^{-\gamma\abs{\lambda/2}}\left|\int_0^{1/2}y^n\eta(y)\,dy\right| \\
&=8C2^{-(n+1)N}e^{\gamma}e^{-\gamma K_0\abs{\mu'-\mu}}\left|\int_0^{1/2}y^n\eta(y)\,dy\right|.
\end{align*}
Combining the two estimates above, we have
\begin{align*}
|2^k\left\langle\Upsilon_k,\psi_{k,K_0\mu}\right\rangle|
&=2^{N(-s+1/q)}\abs{\sum_{\mu'=0}^{2^k-1}2^k\left\langle \eta_{k,K_0\mu'},\psi_{k,K_0\mu}\right\rangle}\\
&\geq 2^{N(-s+1/q)}\left(2^k\left|\left\langle \eta_{k,K_0\mu},\psi_{k,K_0\mu}\right\rangle\right|-\abs{\sum_{\mu'\neq\mu}2^k\left\langle \eta_{k,K_0\mu'},\psi_{k,K_0\mu}\right\rangle}\right)\\
&\geq 2^{N(-s+1/q)}2^{-(n+1)N}\left|\int_0^{1/2}\eta(y)y\,dy\right|\left(\abs{\Tilde{A}}-8C\sum_{\mu'\neq\mu}e^{\gamma}e^{-\gamma K_0\abs{\mu'-\mu}}\right)
\end{align*}
Similarly, by using triangle inequality, we have
\begin{align*}
|2^k\left\langle\Upsilon_k,\psi_{k,K_0\mu}\right\rangle|
&\leq 2^{N(-s+1/q)}\left(2^k\left|\left\langle \eta_{k,K_0\mu},\psi_{k,K_0\mu}\right\rangle\right|+\abs{\sum_{\mu'\neq\mu}2^k\left\langle \eta_{k,K_0\mu'},\psi_{k,K_0\mu}\right\rangle}\right)\\
&\leq 2^{N(-s+1/q)}2^{-(n+1)N}\left|\int_0^{1/2}\eta(y)y\,dy\right|\left(\abs{\Tilde{A}}+8C\sum_{\mu'\neq\mu}e^{\gamma}e^{-\gamma K_0\abs{\mu'-\mu}}\right)\\
\end{align*}
We choose $K_0$ large enough so that 
\[\bigg|4C\sum_{\mu'\neq\mu}e^{\gamma}e^{-\gamma K_0\abs{\mu'-\mu}}\bigg|\leq \frac{|\Tilde{A}|}{2},\]
which gives us the desired result.
\end{proof}

\begin{prop}
\label{convolution bound}
For $x\in J_{k,K_0\mu}$ and $\mu\neq\mu'$, we have that \[\abs{\phi_k*\psi_{k,K_0\mu'}(x)}\leq Ce^{7\gamma/8}e^{-\gamma K_0\abs{\mu-\mu'}}.\]
\end{prop}

\begin{proof}
\[\abs{\phi_k*\psi_{k,K_0\mu'}(x)}=\left|\int\phi_k(y)\psi(2^k(x-y)-K_0\mu')\,dy\right|
=\left|\int \phi(y)\psi(x_1-y-K_0\mu')\,dy\right|\]
where $x_1=2^kx\in J_{0,K_0\mu}\subset K_0\mu+[1/4,3/4]$. We observe that
\[\abs{x_1-y-K_0\mu'}\geq \abs{K_0(\mu-\mu')}-\abs{x_1-y-K_0\mu}\geq\abs{K_0(\mu-\mu')}-7/8.\]
Combining this with the fact that \[\abs{\psi(x)}\leq Ce^{-\gamma\abs{x}},\]
we obtain
\[\abs{\phi_k*\psi_{k,\mu'}(x)}\leq C\int \abs{\phi(y)}e^{-\gamma K_0{\abs{\mu-\mu'}}}e^{7\gamma/8}\leq Ce^{7\gamma/8}e^{-\gamma K_0\abs{\mu-\mu'}}.\]
\end{proof}

\section{Lower Bounds for the Non-Endpoint Case}
\label{sec lower bounds}
In this section, we prove the following, which can be interpreted as a quantitative version of Theorem \ref{main qual thm} for the non-endpoint case.
\begin{thm}
\label{main thm quant}
Let $\Lambda>10$ and let $\gamma_*(\fspq)$ be as defined in \eqref{eq lower proj numbers}. 
\begin{enumerate}
    \item \label{main thm quant, part 1} For $1<p<q<\infty$, $1/q+n<s<1/p+n$, we have \[\gamma_*(\fspq, \Lambda)\gtrsim_{p,q,s} \Lambda^{s-1/q-n}.\]
    \item \label{main thm quant, part 2} For $1<q<p<\infty$, we have $-1/p'-n<s<-1/q'-n$, \[\gamma_*(\fspq, \Lambda)\gtrsim_{p,q,s} \Lambda^{-1/q'-n-s}.\]
\end{enumerate}
\end{thm}
In other words, the magnitude of $\mathcal{G}(\fspq, A)$ depends on the cardinality of $A$ alone. 

\begin{rem}
\label{remark duality}
The statements for \eqref{main thm quant, part 1} and \eqref{main thm quant, part 2} above are equivalent, by a standard argument using the duality of the Triebel-Lizorkin spaces
\[(\fspq)^*=F^{-s}_{p',q'}.\]
We refer the reader to [\cite{seeger2017haar}, Section 2.3 for the details. Consequently, it suffices to prove only the second assertion above.
\end{rem}
The following proposition is the main ingredient in the proof.
\begin{prop}
\label{prop main}
Let $-1<s\leq -1/q'-n$. Let $f_t$ as in (\ref{test function}). Then there exists a $c>0$ such that 
\[\left(\int_0^1\int_0^1\|T_{t_1}f_{t_2}\|_{F^s_{p,q}}^q\,dt_1dt_2\right)^{1/q}\geq c2^{N(-s-1/q'-n)}.\]
\end{prop}

\begin{proof}
We can rewrite the left hand side of the above inequality as
\[\left(\int_0^1\int_0^1\bigg\|\left(\sum_{k=0}^{\infty}2^{ksq}|\phi_k*T_{t_1}f_{t_2}|^q\right)^{1/q}\bigg\|_{L^p}^q\,dt_1dt_2\right)^{1/q}.\]
Restricting the innermost function to the interval $[-1,K_0]$ and using H\"older's inequality (with $p\geq q$), we can bound the expression above by a positive constant times
\begin{align}
&\left(\int_0^1\int_0^1\bigg\|\left(\sum_{2^k\in A}2^{ksq}|\phi_k*T_{t_1}f_{t_2}|^q\right)^{1/q}\bigg\|_{L^q([-1,K_0])}^q\,dt_1dt_2\right)^{1/q}\nonumber\\
&=\bigg(\sum_{2^k\in A}2^{ksq}\bigg\|\bigg(\int_0^1\int_0^1|\phi_k*T_{t_1}f_{t_2}(x)|^q\,dt_1dt_2\bigg)^{1/q}\bigg\|_{L^q([-1,K_0])}^q\bigg)^{1/q}.
\label{qq norm}
\end{align}
For a fixed $x$ we have
\[\phi_k*T_{t_1}f_{t_2}(x)=2^{-N/q}\sum_{2^j\in A}\sum_{2^l\in A}r_j(t_1)r_l(t_2)2^{-ls}\sum_{\mu=0}^{2^j-1}2^j\langle\Upsilon_l,\psi_{j,K_0\mu}\rangle\phi_k*\psi_{j,K_0\mu}(x).\]
By Khinchine's inequality,
\begin{multline*}
\left (\int_{0}^{1}\int_{0}^{1}\abs{\phi_k*(T_{t_1}f_{t_2})(x)}^q\,dt_1\,dt_2\right )^{1/q} \\ \geq c(q)2^{-N/q}\bigg(\sum_{2^j\in A}\sum_{2^l\in A}
|2^{-ls}\sum_{\mu=0}^{2^j-1}2^j\langle\Upsilon_l,\psi_{j,K_0\mu}\rangle\phi_k*\psi_{j,K_0\mu}(x)|^2\bigg)^{1/2}.    
\end{multline*}
For a given $2^k\in A$, we consider only the terms with $j=k$ and $l=k$, and get
\[\left (\int_{0}^{1}\int_{0}^{1}\abs{\phi_k*(T_{t_1}f_{t_2})(x)}^q\,dt_1\,dt_2\right )^{1/q}\gtrsim 2^{-N/q}\abs{2^{-ks}\sum_{\mu=0}^{2^k-1}2^k\langle\Upsilon_k,\kw\rangle\phi_k*\kw(x)}.\]
Now for $x\in J_{k,K_0\mu}$, we have 
\begin{multline*}
\bigg|\sum_{\mu'=0}^{2^k-1}2^k\langle\Upsilon_k,\kww\rangle\phi_k*\kww(x)\bigg| 
\geq \bigg|2^k\langle\Upsilon_k,\kw\rangle\phi_k*\kw(x)\bigg|-\\ \bigg|\sum_{\mu'\neq\mu}2^k\langle\Upsilon_k,\kww\rangle\phi_k*\kww(x)\bigg|,    
\end{multline*}
which, using Proposition \ref{inner product lower bound}, can be bounded below by a positive constant times
\[\frac{|\Tilde{A}|}{2}2^{N(-s+1/q-n-1)}\left|\int_0^{1/2}\eta(y)y\,dy\right|\bigg(|\phi_k*\kw(x)|-4\sum_{\mu'\neq\mu}|\phi_k*\kww(x)|\bigg).\]
An application of Proposition \ref{convolution bound} to the second term in the brackets then yields
\begin{align*}
&\bigg|\sum_{\mu'=0}^{2^k-1}2^k\langle\Upsilon_k,\kww\rangle\phi_k*\kww(x)\bigg| \\
&\gtrsim 2^{N(-s+1/q-n-1)}\left|\int_0^{1/2}\eta(y)y\,dy\right|\bigg(|\phi_k*\kw(x)|-4C\sum_{\mu'\neq\mu}|e^{7\gamma/8}e^{-\gamma K_0\abs{\mu-\mu'}}|\bigg)\\
&\gtrsim c_02^{N(-s+1/q-n-1)}\left|\int_0^{1/2}\eta(y)y\,dy\right|, \text{ for $c_0$ as defined in (\ref{convol big}) and for $K_0$ sufficiently large.}
\end{align*}
Continuing with the proof, we can bound (\ref{qq norm}) below by
\begin{align*}
&\bigg(\sum_{2^k\in A}2^{ksq}\bigg\|\bigg(\int_0^1\int_0^1|\phi_k*T_{t_1}f_{t_2}(x)|^q\,dt_1dt_2\|\bigg)^{1/q}\bigg\|_{L^q([-1,K_0])}^q\bigg)^{1/q}\\
&\gtrsim \bigg(\sum_{2^k\in A}2^{ksq}\sum_{\mu=0}^{2^k-1}\int_{J_{k,\mu}}\bigg[2^{-N/q}2^{-ks}\bigg|\sum_{\mu'=0}^{2^k-1}2^k\langle\Upsilon_k,\kww\rangle\phi_k*\kww(x)\bigg|\bigg]^q\,dx\bigg)^{1/q}\\
&\gtrsim 2^{N(-s+1/q-n-1)}\left|\int_0^{1/2}\eta(y)y\,dy\right|\bigg(\sum_{2^k\in A}2^{-N}\sum_{\mu=0}^{2^k-1}|J_{k,K_0\mu}|\bigg)^{1/q}\\
&\gtrsim 2^{N(-s+1/q-n-1)},
\end{align*}
where we have used (\ref{testbig}) and (\ref{set freq}), and the fact that $|J_{k,K_0\mu}|\gtrsim 2^{-k}$ in the last step.
\end{proof}
\subsection*{Growth of \texorpdfstring{$\gamma_*(\Lambda)$}{} for \texorpdfstring{$s<-1/q'-n$}{}} We take $A$ as in \eqref{set freq}. Let $f_t$ be as in \eqref{test function}, so that $\|f_t\|\lesssim 1$. By Proposition \ref{prop main}, there exist $t_1,t_2$ in $[0,1]$ so that
\[\|T_{t_1}f_{t_2}\|_{\fspq} \gtrsim 2^{N(-s-1/q'-n)}.\]
Hence, 
\[\|T_{t_1}\|_{\fspq\rightarrow\fspq} \gtrsim c_{p,q,s}2^{N(-s-1/q'-n)}.\]
Now we let
\begin{equation}
\label{plus minus sets}
E^{\pm} :=\{\kw :2^k\in A, r_j(t_1)=\pm 1, \mu=0,\dots, 2^k-1\}.    
\end{equation}
Then we have 
\[T_{t_1}=P_{E^+}-P_{E^-}\]
and we conclude that at least one of $P_{E^+}$ or $P_{E^-}$ has operator norm bounded below by $c_{p,q,s}2^{N(-s-1/q'-n)}$. Since $SF(E^{\pm})\subset A$, we get
\[\mathcal{G}(\fspq, A)\gtrsim 2^{N(-s-1/q'-n)}\]
for $s<-1/q'-n$ and the asserted lower bound for $\mathcal{G}(\fspq, A)$ follows in this range.

\begin{rem}
Like the corresponding argument in [\cite{seeger2017haar}, the above proof is probabilistic in nature. In [\cite{seeger2017lower}, Seeger and Ullrich explicitly constructed subsets of the Haar system for which the corresponding projections have large operator norms. It might be of interest to try to adapt this deterministic approach to the case of orthogonal spline wavelets as well.
\end{rem}

\section{Lower Bounds for the Endpoint Case}
\label{sec end point}
In this section we prove the lower bounds for the endpoint cases $s=1/q+n$ and $s=-1/q'-n$. We still have failure of unconditional convergence here, but with a new phenomenon: the growth rate $\mathcal{G}(F^{n+1/q}_{p,q}, A)$ also depends upon the density of the set $\log_2(A)=\{k:2^k\in A\}$ on intervals of length $\log_2(\#A)$.  We define for any $A$ with $\#A\geq 2$, 
\[\underline{\mathcal{Z}}(A)=\text{min}_{2^m\in A}\#\{k\in \log_2(A):|k-m|\leq \log_2(\#A)\}.\]
Then the following is the analog of Theorem \ref{main thm quant} for the endpoint cases:
\begin{thm}
\label{main thm quant end point}
Let $A=\{2^n:n\geq 0\}$ be a set of large enough cardinality.
\begin{enumerate}
    \item \label{main thm end point quant, part 1} For $1<p<q<\infty$, \[\mathcal{G}(F^{n+1/q}_{p,q},A)\gtrsim_{p,q} \log_2(\#A)^{1/q}\underline{\mathcal{Z}}(A)^{1-1/q}.\]
    \item \label{main thm end point quant, part 2} For $1<q<p<\infty$, \[\mathcal{G}(F^{-n-1/q'}_{p,q},A)\gtrsim_{p,q} \log_2(\#A)^{1-1/q}\underline{\mathcal{Z}}(A)^{1/q}.\]
\end{enumerate}
\end{thm}
We can re-frame the above in terms of the lower wavelet projection numbers. 
\begin{cor}
For $\Lambda\geq 4$, we have
\begin{enumerate}
    \item For $1<p<q<\infty$, \[\gamma_*(F^{n+1/q}_{p,q},\Lambda)\gtrsim_{p,q} \log_2(\Lambda)^{1/q}.\]
    \item For $1<q<p<\infty$, \[\gamma_*(F^{-n-1/q'}_{p,q},\Lambda)\gtrsim_{p,q} \log_2(\Lambda)^{1-1/q}.\]
\end{enumerate}
\end{cor}
By Remark \ref{remark duality}, it suffices in this case as well to prove only the second assertion of  Theorem \ref{main thm quant end point}. Let $N$ be such that $4^N\leq 8^{N-1}\leq \#A\leq 8^N$. Using the definition of $\underline{\mathcal{Z}}(A)$, we can find $M_N$ disjoint intervals $I_i=(n_i-3N,n_i+3N)$ with midpoints $n_i\in \log_2(A)$ ($1\leq i\leq M_N$) and $M_N\geq 8^{N-1}/6N\geq 4^N$, such that each $I_i$ contains at least $\underline{\mathcal{Z}}(A)$ points in $\log_2(A)$. By a pigeonholing argument, each $I_i$ contains a subinterval $\Tilde{I_i}$ of length $N$ with at least $\underline{\mathcal{Z}}(A)/6$ points in $\log_2(A)$. The upshot is that we have essentially reduced our problem to proving the following:
\begin{thm}
\label{final thm}
Let $\#A\geq 4^N$. Suppose there exist $4^N$ disjoint intervals $I_{\kappa} (1\leq\kappa\leq 4^N)$, each of length $N$, with $I_{\kappa}\cap\log_2(A)\neq \emptyset$. Let
\begin{equation}
    Z=\frac{1}{4^N}\sum_{\kappa=1}^{4^N}\#(I_{\kappa}\cap A).
\end{equation}
Then, for $q\leq p<\infty$, we have that
\begin{equation}
\label{eq end point growth}
\mathcal{G}(F^{-1/q'}_{p,q};A)\geq c(p,q)N^{1-1/q}Z^{1/q}.    
\end{equation}
\end{thm}

In order to show \eqref{eq end point growth} for the endpoint case, we need to construct a suitable family of test functions. To this effect, let $\eta$ denote a $C^{\infty}$ function supported in $(-2^{-5}, 2^{-5})$ satisfying the conditions \eqref{moment cancel test func} and \eqref{testbig}. However, the parity of $\eta$ would be decided later in the argument.

Let
\[\mathfrak{L}=\{b_{\kappa}+N :\kappa=1,2,\ldots,4^n\}\]
and for $\tau=0,1,\ldots,N-1$, let
\[\mathfrak{L}^{N+\tau}=\{b_{\kappa}+\tau :\kappa=1,2,\ldots,4^n\}.\]
Then $\mathfrak{L}^{N+\tau}$ are disjoints sets, each of cardinality $4^N$. Further, for $l\in \mathfrak{L}$, we define
\begin{equation}
\label{eq test end point}
H_{\kappa}(x)=\sum_{\tau=0}^{N-1}2^{(\tau-N)(n+1)}\sum_{\substack{\rho\in \NN:0<2^{N-b_{\kappa}+2}\rho<1}}\eta(2^{b_{\kappa}+\tau}(x-2^{N+2-b_{\kappa}}K_0\rho)).
\end{equation}
Finally, for $t\in [0,1]$, let
\begin{equation}
f_t(x)=\sum_{\kappa=1}^{4^N}r_{b_{\kappa}+N}(t)2^{(b_{\kappa}+N)/q'}H_{\kappa}(x),    
\end{equation}
where $r_j$ denotes the $jth$ Radamacher function with $j\in \NN$.

\begin{lem}
\label{lemma end point test bound}
We have \[\|f_t\|_{F^{-1/q'-n}_{p,q}}\leq C(p,q)N^{1/q}.\]
\end{lem}
\begin{proof}
Define $g_{\tau,t}$ to be
\[g_{\tau,t}=\sum_{l\in \mathfrak{L}^{N+\tau}}2^{l(n+1/q')}\sum_{\tau=0}^{N-1}2^{(\tau-N)(n+1)}\sum_{\substack{\rho\in \NN:\\0<2^{N+\tau+2-l}\rho<1}}r_{l+N-\tau}(t)\eta(2^l(x-2^{N+\tau+2-l}K_0\rho)).\]
Then we can write
\[f_t=\sum_{\tau=0}^{N-1}2^{(\tau-N)/q}g_{\tau,t}.\]
This sets the stage to apply Proposition \ref{prop tech} with $m=N+\tau$. It is clear that the points $2^{N+\tau+2-l}K_0\rho$ are $K_02^{m-l}$ separated. Using \eqref{eqn tech 2} with $\beta_{N+\tau}=2^{(\tau-N)/q}$, we get
\[\|f_t\|_{F^{-1/q'}_{p,q}}\lesssim_{p,q}(\sum_{\tau=0}^{N-1}((2^{(\tau-N)/q})^q2^{-\tau-N}\#(\mathfrak{L}^{N+\tau})))^{1/q}\lesssim N^{1/q}.\]
\end{proof}
For $\kappa=1,2,\ldots, 4^N$, we define
\begin{equation}
    \mathfrak{A}(\kappa)=I_{\kappa}\cap\log_2(A),
\end{equation}
\begin{equation}
\label{eq end point pts}
    \mathfrak{P}(\kappa)=\{(j,\mu):j\in\mathfrak{A}(\kappa),\mu\in 2^{j-b_{\kappa}+N+2}\ZZ,1\leq\mu<2^j \}.
\end{equation}
and
\begin{equation}
    \mathfrak{P}=\cup_{\kappa=1}^{4^N}\mathfrak{P}(\kappa).
\end{equation}
For $t\in [0,1]$, we also define
\[T_tf(x)=\sum_{(j,\mu)\in \mathfrak{P}}r_j(t)2^j\langle f,\psi_{j,K_0\mu}\rangle\psi_{j,K_0\mu}(x).\]

\begin{prop}
\label{prop end point main}
For $q<p<\infty$, there exists $c(p,q)>0$ such that for $K_0$ and $N$ large enough, we have
\begin{equation}
\bigg(\int_{0}^{1}\int_{0}^{1}\|T_{t_1}f_{t_2}\|^q_{\fspqq}\,dt_1\,dt_2\bigg)^{1/q}\geq c(p,q)NZ^{1/q}.
\end{equation}
\end{prop}
\begin{proof}
As in the non-endpoint case, it suffices to show that
\begin{equation}
\bigg(\int_{0}^{1}\int_{0}^{1}\bigg\|\bigg(\sum_{\kappa}\sum_{k\in\mathfrak{A}(k)}2^{kq(-n-1/q')}|\phi_k*T_{t_1}f_{t_2}|^q\bigg)^{1/q}\bigg\|^q_{L^q([-1,K_0])}\,dt_1\,dt_2\bigg)^{1/q}\geq c(p,q)NZ^{1/q}.
\end{equation}
Interchanging integrals and applying Khinchine's inequality, we see that the above follows if we show
\begin{align*}
&\bigg(\sum_{\kappa}\sum_{k\in\mathfrak{A}(k)}2^{kq(-n-1/q')}\bigg\|\bigg(\sum_{j}\sum_{\kappa'}\bigg|\sum_{\mu:(j,\mu)\in\mathfrak{P}}2^j\langle2^{(b_{\kappa'}+N)(n+1/q')}H_{\kappa'},\psi_{j,K_0\mu}\rangle\\
&\phi_k*\psi_{j,K_0\mu}\bigg|^2\bigg)^{1/2}\bigg\|^q_{L^q([-1,K_0])}\bigg)^{1/q}\geq c(p,q)NZ^{1/q}.
\end{align*}
For the two inner summations, we only consider terms with $j=k$ and $\kappa' =\kappa$. Then the left hand side of the above expression is bounded below by
\begin{equation}
\label{eq carry forward}
\bigg(\sum_{\kappa}\sum_{k\in\mathfrak{A}(k)}2^{kq(-n-1/q')}\bigg\|\sum_{\mu:(j,\mu)\in\mathfrak{P}(\kappa)}2^{(b_{\kappa}+N)(n+1/q')}2^k\langle H_{\kappa},\psi_{j,K_0\mu}\rangle
\phi_k*\psi_{j,K_0\mu}\bigg\|^q_{L^q([-1,K_0])}\bigg)^{1/q}.    
\end{equation}
Setting
\[\zeta_{\kappa,\tau,\rho}(x)=\eta(2^{b_{\kappa}+\tau}(x-2^{N+2-b_{\kappa}}K_0\rho)),\]
\noindent we can write
\[\langle H_{\kappa},\psi_{k,K_0\mu}\rangle=\sum_{0\leq\tau\leq n-1}2^{(\tau-N)(n+1)}\sum_{\substack{\rho\in \NN:\\0<2^{N+2-b_{\kappa}}\rho<1}}\langle \zeta_{\kappa,\tau,\rho},\psi_{k,K_0\mu}\rangle.\]
We recall that by \eqref{eq end point pts}, $\mu$ is of the form $\mu=\mu_n=2^{k-b_{\kappa}+N+2}m$ for some $m\in \NN\cup\{0\}$. Hence
\begin{align*}
2^k\langle \zeta_{\kappa,\tau,\rho},\psi_{k,K_0\mu_m} \rangle  
&=2^k\int \eta(2^{b_{\kappa}+\tau}(x-2^{N+2-b_{\kappa}}K_0\rho))\psi(2^kx-K_0\mu_m)\,dx\\
&=\int \eta(2^{b_{\kappa}+\tau-k}(y-2^{N+2-b_{\kappa}+k}K_0(\rho-m))\psi(y)\,dy.
\end{align*}
Setting $\lambda=2.2^{N+2-b_{\kappa}+k}K_0(\rho-m)$, we observe that the range of the above integral is contained in $[\frac{\lambda-1}{2},\frac{\lambda+1}{2}]$, owing to the fact that $\eta$ is supported in $[-2^{-5},2^{-5}]$. Hence, we can use the spline formulations \eqref{left wave form} and \eqref{right wave form} for $\psi$ on $[\frac{\lambda-1}{2},\frac{\lambda}{2}]$ and $[\frac{\lambda}{2},\frac{\lambda+1}{2}]$ respectively, and argue in the same way as in the proof of Proposition \ref{inner product lower bound}. 

In the light of the moment cancellation condition \eqref{moment cancel test func} (for $\eta$) and symmetry of the lower order co-efficients of $\psi$ (Lemma \ref{lemma wavelet prop}, \ref{coeff lower}), the integrals involving the lower degree terms cancel. We then have
\begin{align*}
&\int \eta\bigg(2^{b_{\kappa}+\tau-k}\bigg(y-\frac{\lambda}{2}\bigg)\bigg)\psi(y)\,dy  \\
=& \int_{\frac{\lambda-1}{2}}^{\frac{\lambda}{2}} \eta\bigg(2^{b_{\kappa}+\tau-k}\bigg(y-\frac{\lambda}{2}\bigg)\bigg)A^n_{\lambda-1}\bigg(y-\frac{\lambda}{2}\bigg)^n\,dy + \int_{\frac{\lambda}{2}}^{\frac{\lambda+1}{2}} \eta\bigg(2^{b_{\kappa}+\tau-k}\bigg(y-\frac{\lambda}{2}\bigg)\bigg)A^n_{\lambda}\bigg(y-\frac{\lambda}{2}\bigg)^n\,dy\\
=&2^{-(b_{\kappa}+\tau-k)(n+1)}\bigg[\int_{\frac{-1}{2}}^0 \eta(y)A^n_{\lambda-1}y^n\,dy + \int_{0}^{\frac{1}{2}} \eta(y)A^n_{\lambda}y^n\,dy\bigg].
\end{align*}
Because of the rapid decay of $\psi$, the major contribution comes in the case when $\rho=m$ (in which case $\lambda=0$), provided we choose $\eta$ in a suitable way, so as to prevent unwanted cancellation. It is here that we choose the parity of $\eta$ to our advantage. More precisely, we choose $\eta$ such that
\[
\eta(y)y^n \text{ is }
\begin{cases}
\text{ odd, }\, \text{ if } A^n_{-1} \text{ and } A^n_{0} \text{ are of opposite signs }\\
\text{ even, }\, \text{ if } A^n_{-1} \text{ and } A^n_{0} \text{ are of the same sign.}\\
\end{cases}
\]
Such a choice ensures that for $\lambda=0$, we have
\[\bigg|\int \eta(2^{b_{\kappa}+\tau-k}y)\psi(y)\,dy\bigg|\geq |A^n_0|2^{-(b_{\kappa}+\tau-k)(n+1)}\bigg|\int_{0}^{\frac{1}{2}} \eta(y)y^n\,dy\bigg|>0.\]

When $\lambda\neq 0$, it is a non-zero integer multiple of $K_0$. Hence the exponential decay of the leading co-efficients $A^n_{\lambda-1}$ and $A^n_{\lambda}$ (Lemma \ref{lemma wavelet prop}, \ref{coeff decay}) kicks in, and we get the bound
\[
\bigg|\int \eta\bigg(2^{b_{\kappa}+\tau-k}\bigg(y-\frac{\lambda}{2}\bigg)\bigg)\psi(y)\,dy\bigg|
\leq 8C2^{-(b_{\kappa}+\tau-k)(n+1)}e^{\gamma}e^{-\gamma\abs{\lambda/2}}\left|\int_0^{1/2}\eta(y)y^n\,dy\right| \]
We now combine the two estimates above and use the triangle inequality as in the proof of Lemma \ref{inner product lower bound}. For $K_0$ large enough, this yields
\begin{multline*}
\frac{|A^n_0|}{2}2^{-(b_{\kappa}+\tau-k)(n+1)}\bigg|\int_{0}^{\frac{1}{2}} \eta(y)y^n\,dy\bigg|\leq \bigg|2^k\sum_{\substack{\rho\in \NN:\\0<2^{N+2-b_{\kappa}}\rho<1}}\langle \zeta_{\kappa,\tau,\rho},\psi_{k,K_0\mu_m}\rangle\bigg|\\
\leq 2|A^n_0|2^{-(b_{\kappa}+\tau-k)(n+1)}\bigg|\int_{0}^{\frac{1}{2}} \eta(y)y^n\,dy\bigg|.   
\end{multline*}
Thus
\[\frac{|A^n_0|}{2}N2^{(k-b_{\kappa}-N)(n+1)}\bigg|\int_{0}^{\frac{1}{2}} \eta(y)y^n\,dy\bigg|\leq 2^k\langle H_{\kappa},\psi_{k,K_0\mu}\rangle\leq 2|A^n_0|N2^{(k-b_{\kappa}-N)(n+1)}\bigg|\int_{0}^{\frac{1}{2}} \eta(y)y^n\,dy\bigg|.\]
The intervals $J_{k,K_0\mu}$ (where $J_{k,K_0\mu}$ is as in \eqref{convol big}) are disjoint. Using Proposition \ref{convolution bound} and arguing as in the non-endpoint case (proof of Proposition \ref{prop main}), we conclude that for $K_0$ large enough, \eqref{eq carry forward} is bounded below by
\begin{align*}
c\bigg(\sum_{\kappa}&\sum_{k\in \mathfrak{A}(\kappa)}2^{kq(-n-1/q')}\\
&\sum_{\substack{m\in \NN:\\0<2^{k+N+2-b_{\kappa}}m<2^k}}\int_{J_{k,\mu_m}}\bigg|2^{(b_{\kappa}+N)(n+1/q')}N2^{(k-b_{\kappa}-N)(n+1)}\int_{0}^{\frac{1}{2}} \eta(y)y^n\,dy\bigg|^q\, dx\bigg)^{1/q}.    
\end{align*}
The measure of $\cup_{m\in \NN:0<2^{k+N+2-b_{\kappa}}m<2^k}J_{k,\mu_m}$ is about $2^{b_{\kappa}-N-k}$. Hence, the above expression is bounded below by
\begin{align*}
&c'\bigg(\sum_{\kappa}\sum_{k\in \mathfrak{A}(\kappa)}2^{kq(-n-1/q')}2^{b_{\kappa}-N-k}
\bigg(2^{(b_{\kappa}+N)(n+1/q')}N2^{(k-b_{\kappa}-N)(n+1)}\bigg)^q\bigg)^{1/q}\\
&\gtrsim \bigg(\sum_{\kappa}\sum_{k\in \mathfrak{A}(\kappa)}2^{-2N}N^q\bigg)^{1/q}\gtrsim NZ^{1/q}.
\end{align*}
This finishes our proof.
\end{proof} 
Finally, we have all the ingredients ready to prove Theorem \ref{final thm}.
\begin{proof}
By Proposition \ref{prop end point main}, there exist $t_1,t_2\in [0,1]$ such that $\|T_{t_1}f_{t_2}\|_{\fspqq}\gtrsim NZ^{1/q}$ but $\|f_{t_2}\|_{\fspqq}\lesssim N^{1/q}.$ Consequently, \[\|T_{t_1}\|_{\fspqq}\gtrsim c_{p,q}N^{1-1/q}Z^{1/q}.\]
Defining $E^{\pm}$ as in \eqref{plus minus sets}, we get
\[\max_{\pm}\|P_{E^{\pm}}\|_{\fspqq}\geq \frac{c_{p,q}}{2}N^{1-1/q}Z^{1/q}.\]
Thus $\mathcal{G}(\fspqq,A)\gtrsim N^{1/q'}Z^{1/q}.$
\end{proof}

\bibliography{main}
\end{document}